\documentclass[manyauthors]{fundam}
\usepackage{url} 
\usepackage[ruled,lined]{algorithm2e}
\usepackage{graphicx}
\usepackage{color}

\newtheorem{observation}{Observation}
\newtheorem{prelem}{{\bf Theorem}}

\newenvironment{oldtheorem}
{\begin{prelem}{\hspace{-0.5em}{\bf}}}{\end{prelem}}

\begin{document}

\setcounter{page}{1}
\publyear{22}
\papernumber{2139}
\volume{188}
\issue{1}

   \finalVersionForARXIV

\title{On Finding Hamiltonian Cycles in Barnette Graphs}

\author{Behrooz Bagheri Gh.\thanks{Research supported in part by FWF
                       Project P27615-N25.}\thanks{Address for correspondence: Algorithms and Complexity Group,
                   Vienna University of Technology,  Favoritenstrasse 9-11, 1040 Vienna, Austria. \newline \newline
                    \vspace*{-6mm}{\scriptsize{Received October 2020; \ accepted  August  2021.}}}
\\
Algorithms and Complexity Group\\
Vienna University of Technology\\
Favoritenstrasse 9-11, 1040 Vienna, Austria\\
behrooz@ac.tuwien.ac.at\\
\and Tomas Feder\\
Computer Science Department \\
Stanford University \\
Stanford, California 94305, USA\\
tomas@theory.stanford.edu\\
\and   Herbert Fleischner$^*$\\
Algorithms and Complexity Group\\
Vienna University of Technology\\
Favoritenstrasse 9-11, 1040 Vienna, Austria
fleischner@ac.tuwien.ac.at\\
\and Carlos Subi\\
carlos.subi@hotmail.com}

\maketitle

\runninghead{B. Bagheri Gh. et al.}{On Finding Hamiltonian Cycles in Barnette Graphs }

\begin{abstract}
In this paper we deal with hamiltonicity in planar cubic graphs $G$
having a facial $2-$factor $\mathcal{Q}$ via (quasi) spanning trees of
faces in $G/\mathcal{Q}$ and study the algorithmic complexity of finding
such (quasi) spanning trees of faces. Moreover, we show that
if Barnette's Conjecture is false, then hamiltonicity in $3-$connected
planar cubic bipartite graphs is an NP-complete problem.
\end{abstract}


\section{Introduction}

Our joint paper~\cite{Bagheri} can be considered as the
point of departure for the subsequent discussion and results of this
paper. {We start with} a few historical remarks.
In $1884$, Tait  conjectured that every cubic
$3-$connected planar graph is
hamiltonian~\cite{Tait}.
And Tait knew that the validity of his conjecture would yield a simple
proof of the Four Color Conjecture. On the other hand, the Petersen graph
is the smallest non-planar $3-$connected cubic graph which is not hamiltonian,~\cite{Petersen}.
Tait's Conjecture
 was disproved by Tutte in $1946$, \cite{Tutte1946}. However, none of the known counterexamples of Tait's Conjecture is
bipartite. Tutte himself conjectured that every cubic $3-$connected
bipartite graph is hamiltonian~\cite{Tutte1971}, but this was shown to be
false by the construction of a counterexample,
the Horton graph~\cite{Horton}.  Barnette  proposed a
combination of Tait's and Tutte's Conjectures implying that every counterexample
to Tait's conjecture is non-bipartite. {However, it is a well-known fact that hamiltonicity in planar cubic graphs is an
NP-complete problem. This implies that the existence of an $A-$trail in
plane eulerian graphs is also an NP-complete problem even if restricted
to planar $3-$connected eulerian graphs (see Definition~\ref{DEF:A-trail}.(i) below).}

\medskip \noindent
{\bf Barnette's Conjecture}~\cite{Barnette}
{\it Every $3-$connected cubic planar bipartite graph is hamiltonian.}\\

We denote $3-$connected cubic planar bipartite graphs as {\sf Barnette graphs}.
Holton, Manvel and McKay showed in~\cite{Holton}  that Barnette graphs with up to $64$ vertices are hamiltonian.
The conjecture
also holds for the infinite family of Barnette graphs where all faces are either quadrilaterals or  hexagons, as shown by
Goodey~\cite{Goodey}. However, it is NP-complete to decide whether
a $2-$connected cubic planar bipartite graph  is hamiltonian~\cite{Takanori}.

{We note that \cite{Feder} can essentially be viewed in part as a preliminary version of \cite{Bagheri} and the current paper, with \cite{Bagheri} focusing on graph theoretical results in \cite{Feder}, whereas the current paper's focus are mainly algorithmic and complexity considerations as developed in \cite{Feder}. However, additional results were developed to put \cite{Feder} in a more general perspective.}

{We outline the structure of this paper as follows.}

\medskip
{Section 2 of this paper starts with listing several known results from \cite{Bagheri} and other papers, and several definitions; they can be viewed as the basis for this paper.}

\medskip
{In Section 3 we first prove some structural results leading to a decision in polynomial time whether a Barnette graph with certain properties has a hamiltonian cycle of a special type (Corollary \ref{CR:2}).}

\medskip
{In Section 4 we first prove NP-completeness of the existence of certain types of spanning trees of faces (see Definition \ref{DEF:QST}), and subsequent corollaries. Finally it is shown that if Barnette's Conjecture is false, then hamiltonicity in Barnette graphs is an NP-complete problem.}

\section{Preliminary discussion}

As for the terminology used in this paper we follow~\cite{Bondy} unless stated explicitly otherwise.
In particular, the subset $E(v)\subseteq E(G)$ denotes the set of edges incident to $v\in V(G)$.
For definitions we refer to~\cite{Bagheri} but for completeness' sake we repeat some of them. {Moreover, we present several known results (Theorems A-E) which can be viewed as the frame in which the results of this paper take place.}

\begin{definition}\label{DEF:A-trail}\hfil\hbox{}\break

\vspace*{-6mm}
\begin{itemize}
\item[(i)] Let $H$ be a $2-$connected eulerian plane graph. An eulerian trail $L$ in $H$ is an {\sf $A-$trail}
if   any two consecutive edges of $L$  belong to a face boundary.
\item[(ii)]  An $A-$trail $L$ in an eulerian triangulation of the plane is called {\sf non-separating} if for every face boundary $T$  at least two edges of $E(T)$ are consecutive in $L$.
\item[(iii)] An {\sf $A-$partition} of $H$ is a vertex partition  $V_L(H)=\{V_1,V_2\}$  induced by
an $A-$trail $L=e_1e_2\ldots e_m$    as follows.  Consider  a
$2-$face-coloring of $H$ with colors $1$ and $2$. For every vertex $v$ of $H$, $v\in V_i$ if and only if  there is $j\in \{1,\ldots,m-1\}$ such that $v\in V(e_j)\cap V(e_{j+1})$ and the face containing $e_j$ and $e_{j+1}$ in its boundary is colored $3-i$,  $i=1,2$,
where $V(e)$ is the set of vertices incident to the edge $e$.
\end{itemize}
\end{definition}
{There is a close conection between hamiltonian cycles in Barnette graphs and A-trails in their dual.}

\begin{oldtheorem}~$(${\rm\cite[Theorem VI.71]{Fleischner}}$)$\label{dualAtrail}
A $2-$connected plane cubic bipartite graph has a hamiltonian cycle if and only if its dual graph has a non-separating $A-$trail.
\end{oldtheorem}

\begin{definition}\label{DEF:Radial}\hfil\hbox{}\break

\vspace*{-7mm}
\begin{itemize}
\item[(i)] Suppose $H$ is a $2-$connected plane  graph. Let  $\mathcal{F}(H)$ be
the set of faces of $H$. The {\sf radial graph} of $H$ denoted by
$\mathcal{R}(H)$ is a bipartite graph
with the vertex bipartition $\{V(H),\mathcal{F}(H)\}$
such that $xf\in E(\mathcal{R}(H))$ if and only if $x$ is a vertex in the boundary of
$F\in \mathcal{F}(H)$ corresponding to $f\in V(\mathcal{R}(H))$.\newline
\item[(ii)]
 Let $U\subseteq V(H)$ and let $\mathcal{T}\subset \mathcal{F}(H)$ be a
set of bounded faces of $H$.
The {\sf restricted radial graph}
$\mathcal{R}(U,\mathcal{T})\subset \mathcal{R}(H)$ is defined  by $\mathcal{R}(U,\mathcal{T})=\langle U\cup \mathcal{T}\rangle_{\mathcal{R}(H)}$.
\end{itemize}
\end{definition}

{For the next definition} let $H$ be a  $2-$connected plane graph, let $U\subseteq V(H)$
and let $\mathcal{T}\subset \mathcal{F}(H)$ be a set of bounded faces
whose boundaries are pairwise edge-disjoint
and such that every vertex of $H$ is contained in some element of $\mathcal{T}$.
 We define a subgraph $H_{\mathcal{T}}$ of $H$  by
$H_{\mathcal{T}}=H[ \cup_{F\in \mathcal{T}}E(F)]$.

\begin{definition}\label{DEF:QST}
If  $\big|\big\{F\in \mathcal{T}\ :\ x\in V(F)\big\}\big|=\frac{1}{2} \deg_H(x)$
for every $x\in V(H)\setminus U$,
and if $\mathcal{R}(U,\mathcal{T})$ is a tree, then we call $H_{\mathcal{T}}$
a {\sf quasi spanning tree of faces} of $H$, and the vertices in
$U\ (V(H)\setminus U)$ are called {\sf proper} $(${\sf quasi}$)$ vertices.
If $U=V(H)$, then $H_{\mathcal{T}}$ is called a {\sf spanning tree of faces}.
\end{definition}

{In other words, a spanning tree of faces is a spanning bridgeless cactus whose cycles are face boundaries.}

\begin{definition}\label{DEF:Leapfrog}
	The {\sf leapfrog extension} $Lf(G)$ of a $2-$connected cubic plane
	graph $G$ is the cubic plane graph obtained from $G$  by replacing every $v\in V(G)$ by
	a hexagon $C_6(v)$, with $C_6(v)$ and $C_6(w)$ sharing an edge if and
	only if $vw\in E(G)$; and these hexagons are faces of $Lf(G)$.
\end{definition}

We note in passing that we call leapfrog extension
what is called in other papers  vertex envelope, or leapfrog construction, or leapfrog operation, or  leapfrog transformation (see e.g.~\cite{FleischnerEnvelope,Fowler,Kardos,Yoshida}).
\vskip 0.4cm
If a plane graph has a face-coloring with color set $X$, the faces of color $x\in X$ will be called $x-$faces.

\begin{observation}\label{obs:A-trail}
We observe that if $H$ is a  plane eulerian graph with $\delta(H)\ge 4$
having an $A-$trail $L$, then  $L$ defines
uniquely a quasi spanning tree of faces. {Conversely, a $($quasi$)$ spanning tree of faces $H_{\mathcal{T}}$ defines
uniquely an $A-$trail in $H_\mathcal{T}$ $($where $\mathcal{T}$ is defined as in the paragraph preceding Definition~$\ref{DEF:QST})$.}
\end{observation}

\begin{proof}{Start with a
$2-$face-coloring of $H$ with colors $1$ and $2$, suppose the outer face of $H$ is colored $1$. To show the first part of Observation 1, let $V_L(H)=\{V_1, V_2\}$ be the  $A-$partition of $V(H)$ induced by $L$ $($see {\rm Definition~\ref{DEF:A-trail}.(iii))}.
Now, the set of all $2-$faces  defines a quasi spanning tree of
faces $H_{\mathcal{T}}$ with $V_1$ being the set of all quasi vertices of
$H_{\mathcal{T}}$. The second part of Observation 1 follows similarly.
}\end{proof}

The aforementioned relation between the concepts of $A-$trail  and
(quasi) spanning tree of faces is not a coincidence. In fact, it had been shown
(\cite[pp.~$VI.112-VI.113$]{Fleischner} -- {see also Theorem A})
that
\vskip 0.4cm
$\bullet$ {\it Barnette's Conjecture is true if and only if every simple $3-$connected eulerian
triangulation of the plane admits an $A-$trail.}\\

We point out, however, that the concept of (quasi) spanning tree of faces
is a somewhat more general tool to deal with hamiltonian cycles in
plane {cubic} graphs, than the concept of $A-$trails. We are thus focusing our considerations below on the
complexity of the existence of $A-$trails and (quasi) spanning trees
of faces in plane (eulerian) graphs.

Parts of this paper are the result of extracting
some results and their proofs of~\cite{Feder};
 they have not been published in a refereed journal yet.
Moreover, we relate some of the results of this paper to the theory of
$A-$trails, as developed in~\cite{Fleischner}.

Next we list some results  of the preceding joint paper \cite{Bagheri}
which will be essential for the current paper.
In general, when we say that $F$ is an $\mathcal{X}-$face ($\mathcal{X}^{c}-$face), we mean that
  $F\in \mathcal{X}\ (F\notin \mathcal{X})$.

Consider   a $3-$connected cubic plane graph $G$
with a facial $2-$factor $\mathcal{Q}$.
Given a quasi spanning tree of faces $H_{\mathcal{T}}$ in the reduced
graph $H=G/\mathcal{Q}$, we assume the outer face is not in $\mathcal{T}$, and traverse the $A-$trail of $H_{\mathcal{T}}$ (see the second part of Observation~\ref{obs:A-trail}), to obtain a hamiltonian
cycle  $C$  in $G$ such that the faces of $\mathcal{Q}$ corresponding to the proper vertices of $H_{\mathcal{T}}$ lie inside of  $C$  whereas the faces of $\mathcal{Q}$ corresponding to quasi vertices of $H_{\mathcal{T}}$ lie outside of  $C$. In fact the following is true.

\begin{oldtheorem}~$(${\rm\cite[Proposition~1]{Bagheri}}$)$
\label{PR:1}
Let $G$ be   a $3-$connected cubic plane graph $G$
with a facial $2-$factor $\mathcal{Q}$.
  Then, the reduced
graph $H=G/\mathcal{Q}$ has a quasi spanning tree of faces, $H_{\mathcal{T}}$, with
 the outer face not
in $\mathcal{T}$, if and only if $G$
has a hamiltonian cycle $C$ with  the outer $\mathcal{Q}^{c}-$face outside of $C$,
with all $\mathcal{Q}-$faces corresponding to proper vertices
of $H_{\mathcal{T}}$
inside of $C$, with
all $\mathcal{Q}-$faces corresponding to quasi vertices
of $H_{\mathcal{T}}$
outside of $C$, and such that no two $\mathcal{Q}^{c}-$faces sharing an edge are both
inside of $C$.
\end{oldtheorem}

\begin{oldtheorem}~$(${\rm\cite[Corollary~7]{Bagheri}}$)$\label{cor:4con-eulerian-triangulation}
Every simple $4-$connected eulerian triangulation of the plane has a
quasi  spanning tree of faces.
\end{oldtheorem}

{Note, however, that it is an open problem whether such triangulations have an $A$--trail.}
\vskip 0.4cm
{Additionally,} two old results are listed below.
\begin{oldtheorem}~$(${\rm\cite[Corollary~3]{Andersen95}}$)$\label{CR:A-Tr.NP-C}
The  problem of deciding whether a planar eulerian graph admits an $A-$trail remains NP-complete for $3-$connected graphs having only $3-$cycles and $4-$cycles as face boundaries, and for which all faces with $4-$cycles as boundaries have the same color in the $2-$face-coloring.
\end{oldtheorem}

In contrast,  Andersen et al. in~\cite{Andersen98} gave a polynomial algorithm for
finding $A-$trails in simple $2-$connected outerplane eulerian graphs.

\begin{oldtheorem}~$(${\rm\cite[Theorem~23]{FleischnerEnvelope}}$)$\label{TH:Lf.NP-C}
The decision problem of whether the leapfrog extension of a plane cubic graph
with multiple edges is hamiltonian is NP-complete.
\end{oldtheorem}

\section{Graph theoretical results and polynomially solvable problems}\label{sec2}

In proving   Propositions $3$ and $6$ and Theorem $11$
in~\cite{Bagheri}, we used implicitly some
algorithms to construct a (quasi) spanning tree of faces. In all of them,
we find some triangular face such that the  graph
resulting from the
contraction of this face, still satisfies the
hypothesis of the respective result.
 By repeating this
process, finally the contracted faces together  with a
special face form a (quasi) spanning tree of faces.
 Note that it is possible to identify the contractible faces in linear time, since every simple plane graph has $O(n)$ faces, where $n$ is the order of a graph  under consideration.
 Therefore, our
algorithms for finding a quasi spanning tree of faces   in~\cite{Bagheri}
are  polynomial.

We show next that one can decide in polynomial time whether the reduced graph $H$
with   a set $\mathcal{D}$
 of
 edge-disjoint  faces in $H$ has a spanning tree of
faces in $\mathcal{D}$ provided every face boundary of $H$ is a digon or a triangle.
\vskip 0.4cm

{\bf The Spanning Tree Parity Problem:} {\it
Given a graph $G$ and a collection of {disjoint} pairs of edges,
$\{\{e_i,f_i\}\ |\ i=1,\ldots, k\}$. {\sf The Spanning Tree Parity Problem (STPP)}
asks whether $G$ has a spanning tree $T$ satisfying $|\{e_i,f_i\}\cap E(T)|\in \{0,2\}$, for each $i=1,\ldots,k$.}
\vskip 0.4cm
Note that the STPP is
solvable in polynomial time (see~{\rm\cite{Gabow,Lovasz}}).

\begin{theorem}\label{TH:3}
Let $G$ be a $3-$connected cubic plane graph
having a facial $2-$factor $\mathcal{Q}$ and $H=G/\mathcal{Q}$.
 Let $\mathcal{D}$ be a set
 of
 edge-disjoint faces in $H$ such that $\mathcal{D}$ covers all of $V(H)$ and such that
all faces in $\mathcal{D}$ are either digons or triangles. Then we can decide in polynomial time whether $H$ has a spanning
tree of faces in $\mathcal{D}$, yielding a hamiltonian cycle for $G$, by a spanning tree
parity algorithm.
\end{theorem}
\begin{proof}{ Construct a graph $H^{'}$ related to $H$ as follows. $V(H^{'})=V(H)$. If $xyx$ is a digon in $\mathcal{D}$, then let  $xy$ be an edge in $H^{'}$. If $xyzx$ is a triangle
in $\mathcal{D}$, then put edges $xy$ and $yz$ in $H^{'}$
(the naming of the  vertices of the triangle with the symbols $x,y,z$ is arbitrary but fixed). A spanning tree of faces in $\mathcal{D}$ for $H$ then corresponds
to a  spanning tree $T$ in $H^{'}$
satisfying $|\{xy,yz\}\cap E(T)|\in \{0,2\}$, for each
 triangle $xyzx$ in $\mathcal{D}$. Thus, these conditions on pairs of edges in $H^{'}$
transform the problem of finding a spanning tree of faces
 in $\mathcal{D}$ for $H$, yielding a hamiltonian cycle for $G$ by Theorem~\ref{PR:1}, equivalently in polynomial time into an STPP in $H^{'}$.
}\end{proof}

If $\mathcal{D}$ contains faces with four or more sides, say a face $xyztx$, then we could include three
edges linking these four vertices, say $xy,\ yz,$ and $zt$, and require that a spanning tree  must contain
either all three or none of these three edges. Such a Spanning Tree Triarity Problem (STTP), as we
shall  see later in Theorem~\ref{TH:4}, turns out to be NP-complete.

The following result demonstrates the close relationship between $A-$trails and spanning trees of faces vis-a-vis hamiltonian cycles.

\begin{theorem}\label{PR:3}  Let $G$ be a Barnette graph whose faces are $3-$colored with color set $\{1,2,3\}$ and suppose without loss of generality that the outer face of $G$ is a $3-$face. The following statements are equivalent.

\begin{description}
\itemsep=0.75pt
\item[$(i)$]  $G$ has a hamiltonian cycle $C$ with the $2-$faces  lying inside of $C$, the $3-$faces lying outside of $C$, and $1-$faces on either side;
\item[$(ii)$] the reduced graph $H$ obtained by contracting  the
$1-$faces has an $A-$trail;
\item[$(iii)$] the reduced graph $H^{'}$ obtained by
contracting  the $2-$faces has a spanning tree of $1-$faces;
\item[$(iv)$] the reduced graph $H^{''}$
obtained by contracting  the $3-$faces has a spanning tree of $1-$faces.
\end{description}
\end{theorem}

\begin{proof}
$(i)\Rightarrow(ii):$ Let $T_C$ be a
closed trail in $H$ induced by
hamiltonian cycle $C$ of $G$ having the properties described in $(i)$.  $T_C$ is an eulerian trail, otherwise there are two  faces of $G$ with two different colors $2$ and $3$ lying on one side of $C$.
This  obvious contradiction to $(i)$ guaranties that  $T_C$ is an eulerian trail. Since all $2-$faces ($3-$faces) of $G$ are lying inside (outside) of $C$, for every $1-$face $F_1$ of $G$ we conclude that $E(F_1)\cap E(C)$ is a matching.
Therefore, every pair of consecutive edges  of $T_C$   corresponds to  a path of length  three  in $C$ such that the central edge of this path from one side belongs to
a $1-$face boundary, and from the other side all three edges belong to an $i-$face boundary, $i\in \{2,3\}$.
Thus, any two consecutive edges of $T_C$ belong to a face boundary in $H$ and so $T_C$ is an $A-$trail.

\medskip\noindent
$(ii)\Rightarrow(i):$ The $3-$face-coloring of $G$ induces a $2-$face-coloring in $H$
using colors $2$ and $3$ and such that the outer face of $H$ is a $3-$face. Now it is easy to see that any $A-$trail of $H$ can be transformed into a hamiltonian cycle $C$ of $G$ with the $2-$faces  lying inside of $C$, the $3-$faces lying outside of $C$, and $1-$faces lying on either side.

\medskip\noindent
$(i)\Rightarrow(iii):$
Now consider the  $2-$face-coloring of $H^{'}$ induced by the $3-$face-coloring of $G$. Let $U= V(H^{'})$ be the vertex set corresponding to the $2-$faces. Also, let $\mathcal{T}$ be the set of  $1-$faces
of $H^{'}$ corresponding to the $1-$faces in $int(C)$.

\smallskip
Observe that $G_{int}:=C\cup int(C)$
is a spanning outerplane subgraph of $G$, and that the weak dual
(the subgraph of the dual graph whose vertices correspond to the bounded faces) of
$G_{int}$ is a tree (see~\cite{Fleischner}). Therefore,
$H^{'}_{int}\subset H^{'}$ being the reduced graph of $G_{int}$
after contracting the $2-$faces, is a spanning tree of faces in $H^{'}$.

\medskip\noindent
$(iii)\Rightarrow(i):$   Suppose $H^{'}$ has a spanning tree of
 $1-$faces $H^{'}_{\mathcal{T}}$.
 Then $H^{'}_{\mathcal{T}}$ has a unique $A-$trail which can be transformed    into a hamiltonian cycle $C$ of $G$ such that
 the $2-$faces
 (corresponding to $V(H^{'})$) lie in $int(C)$ and
the  corresponding $3-$faces   lie in $ext(C)$.

\medskip
The equivalence of $(i)$ and $(iv)$ is established analogously by looking at $G_{ext}:=C\cup ext(C)$ which is also an outerplanar graph.
\end{proof}

An application of Theorems~\ref{TH:3} and~\ref{PR:3}
yields the following.

\begin{corollary}\label{CR:2} Let $G$ be a Barnette  graph
with a $3-$face-coloring with color set $\{1,2,
3\}$, and let $H$ be the reduced graph obtained by contracting
 the $1-$faces. Suppose all vertices of $H$ have degree $4$ or $6$.
Then one can decide in polynomial time whether $H$ has an $A-$trail
 which in turn yields a hamiltonian cycle in $G$.
\end{corollary}

\begin{proof}{ Let $H$ be the reduced graph of $G$ obtained by contracting the $1-$faces, and let $H^{'}$ be the reduced graph of $G$ obtained by contracting the $2-$faces instead
of the $1-$faces. Then  each $1-$face of $G$ yields a
digon or triangle in $H^{'}$. By Theorem~\ref{PR:3}, an
$A-$trail  in $H$ corresponds to a spanning tree of $1-$faces in $H^{'}$ and vice versa. Since all
$1-$faces  of $H^{'}$ are either digons or triangles, one can decide
in polynomial time by Theorem~\ref{TH:3} whether such a spanning tree
of $1-$faces exists in $H^{'}$.
}\end{proof}

By Observation~\ref{obs:A-trail}, we have the following theorem
in which we make use of the fact that an
(eulerian) triangulation of the plane admits two interpretations, namely: as the dual of a plane cubic   (bipartite) graph, and as the contraction of a facial (even) $2-$factor $Q$ in $G$ whose faces in $Q^c$ are hexagons.
\begin{theorem}\label{TH:Herbert}
	 Let $G$ be a  Barnette graph and let $\mathcal{F}$ be the set of its faces. Let
$\mathcal{Q}_{\mathcal{F}}$ be the facial $2-$factor of $Lf(G)$
corresponding to $\mathcal{F}$ and let the color classes of the $3-$face-coloring of $Lf(G)$ be denoted by $F_1$, $F_2$, and $F_3$ such that $F_3 = \mathcal{Q}_{\mathcal{F}}$, and thus $F_1$, $F_2$ translates into a $2-$face-coloring of $Lf(G)/\mathcal{Q}_{\mathcal{F}}$  denoting the corresponding sets of faces by $F1$, $F2$ and whose vertex set $($corresponding to $F_3)$ be denoted by $V_3$.
$G^{*}$ denotes the dual of $G$. Then the following is true.
\begin{description}
\itemsep=0.9pt
\item[$(1)$]
 $G^{*} = Lf(G)/\mathcal{Q}_{\mathcal{F}}.$
\item[$(2)$] $G$ is hamiltonian if and only if $Lf(G)$ has a hamiltonian cycle $C$ such that $int(C) = F_1\cup F_3^{'}$ and $ext(C) = F_2\cup F_3^{''}$ where $F_3 = F_3^{'}\dot{\cup} F_3^{''}$. \medskip\\
Statement $(2)$ is equivalent to
\item[$(3)$] $G^{*}$ has a non-separating  $A-$trail if and only if  there is a partition $V_3=V_3^{'}\dot{\cup} V_3^{''}$ such that $Lf(G)/\mathcal{Q}_{\mathcal{F}}$ has a quasi spanning tree of faces containing all of $F1$ and where $V_3^{'}$ is its set of proper vertices and $V_3^{''}$ is its set of quasi vertices.
$(V_3^{'}$ and $V_3^{''}$ are the vertex sets in  $Lf(G)/\mathcal{Q}_{\mathcal{F}}$ corresponding to $F_3^{'}$ and $F_3^{''}$, respectively, - see $(2)$ above$)$.
\end{description}
\end{theorem}

\begin{proof}
By Definition~\ref{DEF:Leapfrog} and definition of the dual graph of a plane graph,  statement $(1)$ is true.

\medskip
Next we show that $(2)$ is true. Assume $G$ has a hamiltonian cycle $C_0=e_1e_2\ldots e_n$ such that $e_i=v_iv_{i+1}$ for $i=1,\ldots,n-1, e_n = v_nv_1$.
Let $e=v_iv_j\in E(G)$ be the edge  corresponding to $e^{'}\in E(C_6(v_i))\cap E(C_6(v_j))\subset E(Lf(G))$, for $1\le i\neq j\le n$ (see Definition~\ref{DEF:Leapfrog}
concerning $C_6(v_i)$).

Now we construct a hamiltonian cycle $C$ in $Lf(G)$ corresponding to $C_0$ as follows.
Consider $C^{\circ}=\{e_i^{'}\ |\ i=1,\ldots,n\}$. If $C_6(v_{i+1})\in F_1$ $(C_6(v_{i+1})\in F_2)$ where $\{e_i^{'},e_{i+1}^{'}\}\subset E(C_6(v_{i+1}))$, add the path in $C_6(v_{i+1})$ connecting the endvertices of $e_i^{'}$ and $e_{i+1}^{'}$  outside (inside, respectively) $C_0$ to $C^{\circ}$, for $i=1,\ldots,n-1$. Then add the path in $C_6(v_{1})$ connecting the endvertices of $e_1^{'}$ and $e_n^{'}$  inside (outside, respectively) $C_0$ to $C^{\circ}$; call the final set thus constructed $C$.
By the construction of $C$, $int(C) = F_1\cup F_3^{'}$ and $ext(C) = F_2\cup F_3^{''}$ where $F_3 = F_3^{'}\dot{\cup}  F_3^{''}$. Since every vertex $v\in V(Lf(G))$  is incident to an adge $e$ which belongs to an $i-$face boundary, $i\in \{1,2\}$, by $F_1\subset int(C)$
and $F_2\subset ext(C)$, we have $C$ traverses the edge $e$ and then $v\in V(C)$. Therefore, $C$ is hamiltonian.
\\
Conversely, it is straightforward to see that a hamiltonian cycle in $Lf(G)$ as described yields a hamiltonian cycle in $G$.
Thus, $(2)$ is true.

\medskip
Theorem~\ref{dualAtrail} emplies that  $(2)$ is equivalent to the left side of $(3)$.

And finally we show that   $(2)$ is equivalent to the right side of $(3)$. Again consider a  hamiltonian cycle in $G$. By $(2)$, there is a hamiltonian cycle $C$ in $Lf(G)$ such that $int(C) = F_1\cup F_3^{'}$ and $ext(C) = F_2\cup F_3^{''}$ where $F_3 = F_3^{'}\dot{\cup} F_3^{''}$.
 Now, let $V_3^{'}$  be the set of vertices in $Lf(G)/\mathcal{Q}_{\mathcal{F}}$ corresponding to $F_3^{'}$. Then, it can be seen easily that  $\mathcal{R}(V_3^{'},F1)$ is a tree;
 and therefore,   $Lf(G)/\mathcal{Q}_{\mathcal{F}}$ has a quasi spanning tree of faces containing all of $F1$ and where $V_3^{'}$ is its set of proper vertices.
 Thus, $(3)$ is true. The converse is true by Theorem~\ref{PR:3}.
\end{proof}

Theorem~\ref{TH:Herbert} puts hamiltonicity in $G$ in a qualitative perspective of the algorithmic complexity regarding quasi spanning trees of faces of a special type in
the reduced graph  of the leapfrog extension of $G$.
In fact, if $\mathcal{G}$ is a class of Barnette graphs where hamiltonicity can be decided in polynomial time, then the same can be said regarding
special types of
quasi spanning trees of faces in the reduced graphs of the leapfrog extensions  of the elements of  $\mathcal{G}$ (as stated in the theorem). For, given a hamiltonian cycle $C_0$ in $G\in \mathcal{G}$, a
non-separating  $A-$trail $L_{C_0}$ in $G^{*}$ can be found in polynomial time which in turn yields a quasi spanning tree of faces in $Lf(G)/\mathcal{Q}_{\mathcal{F}}$ as described in $(3)$, 
also in polynomial time. Compare this with Theorem~\ref{TH:Lf.NP-C} and Theorem~\ref{cor:4con-eulerian-triangulation}.

The following proposition shows that if Barnette's Conjecture is false then there is a particular edge $e$ in some hamiltonian Barnette graph such that every  hamiltonian cycle of that graph contains $e$.

\begin{proposition}\label{Pr:G1}
If there exists a non-hamiltonian Barnette graph, then there exists a hamiltonian Barnette graph $G_1$
with a particular edge $e=uv$  such that $e\in E(C)$ for every hamiltonian cycle $C$ of $G_1$. Furthermore, if $e_1$ and $e_2$ are the two edges other than $e$ incident to $u$ in
$G_1$, then $G_1$ has a hamiltonian cycle $C_i$ traversing $e$ and $e_i$, for $i=1,2$.
\end{proposition}

\begin{proof}{
Suppose $G_0$ is a smallest counterexample to Barnette's Conjecture.

First we construct a hamiltonian Barnette graph $G_1$
with a particular edge $e_0=u_0v_0$  such that $e_0\in E(C)$ for every hamiltonian cycle $C$ of $G_1$.

 Let $Q=wxyzw$ be a facial quadrilateral
 in $G_0$ and let $a_1$ be the third neighbour of $a$ in $G_0$, for $a\in \{w,x,y,z\}$.

Set $G_0^{'}=(G_0\setminus \{w,x,y,z\})\cup \{w_1x_1,y_1z_1\}$ and
$G_0^{''}=(G_0\setminus \{w,x,y,z\})\cup \{w_1z_1,x_1y_1\}$.
Both $G_0^{'}$ and $G_0^{''}$ are planar, cubic and bipartite.

Suppose that $G_0^{'}$ is   $3-$connected. By minimality of $G_0$, the graph $G_0^{'}$ has a hamiltonian
cycle. Furthermore, no hamiltonian cycle of $G_0^{'}$ goes through either the edge $w_1x_1$ or the edge $y_1z_1$; otherwise,
 we can extend this cycle to a hamiltonian cycle  in $G_0$, a contradiction.

  We have thus guaranteed that no
hamiltonian cycle in $G_1=G_0^{'}$ traverses a particular edge $w_1x_1$, and thus every hamiltonian cycle traverses an edge $e_0$
adjacent to $w_1x_1$, as desired. The same conclusions can be drawn if $G_0^{''}$ is $3-$connected.

Suppose now that $G_0^{'}$ and $G_0^{''}$ are both  $2-$connected only. Then there are two edge cuts of size four, $T_1$ and $T_2$, in $G_0$ such that $\{wx,yz\}\subset T_1$ and  $\{wz,xy\}\subset T_2$.

 Removing the vertices $w,x,y,z$ and the remainder of the two edge cuts $T_1$ and $T_2$
separates $G_0$ into four components $R_1, R_2, R_3, R_4$, with the removed edges of $G_0$ including
an edge from $R_i$ to $R_{i+1}$, for $i=1,2,3$, and an edge from
$R_4$ to $R_1$, plus the four edges from the four $R_i$'s  incident to a vertex of $Q$.
That is, each $R_i$  is incident to  three   edges whose
endvertices not in $R_i$ can be identified to a single vertex $r_i$ to obtain a bipartite $R_i^{'}$, since their
three endvertices in $R_i$ are at even distance from each other. For, in the $2-$vertex-coloring of $G_0$, the three
 vertices of degree $2$ of $R_i,\ 1\le i\le 4$, must have the same color;
otherwise, two copies of such $R_i$ could be used to construct a cubic bipartite graph having a bridge. Clearly, $R_i^{'}$
is $3-$connected, cubic, planar, and bipartite, for each $i=1,\ldots,4$.

By minimality of $G_0$ each such   $R_i^{'}$
 has a hamiltonian cycle, yet it is not the case that each of the
three choices of two edges incident to $r_i$ yields a hamiltonian cycle, since otherwise we
could obtain a hamiltonian cycle for $G_0$. Thus one of the three edges
incident to $r_i$ in $R_i^{'}$ must belong to every hamiltonian cycle, thus yielding a Barnette graph $G_1= R_1^{'}$ with an edge $e_0=u_0v_0$ that belongs to every hamiltonian cycle of $G_1$.

If $G_1$ has a hamiltonian cycle $C_i^{'}$ traversing $e_0$ and $e_i^{'}=u_0v_i^{'}$, for $i=1,2$,
then let  $e=e_0$, $C_i=C_i^{'}$, and $e_i=e_i^{'}$, for $i=1,2$. This would complete the proof of  Proposition~\ref{Pr:G1}.  Thus, suppose instead that every hamiltonian cycle in $G_1$ is forced to
traverse $e_1^{'}=u_0v_1^{'}\in E(G_1)$ as well.

Let $C_0 = f_0, f_1,\ldots,f_{n-1}$ where $f_0=e_0$ and $f_1=e_1^{'}$ be a fixed hamiltonian cycle of $G_1$ and let $k$ to be the largest index such that the section $f_0,\ldots,f_{k-1}$ belongs to every hamiltonian cycle of $G_1$, whereas $f_k$ does not belong to every hamiltonian cycle of $G_1$. Such $k$ must exist; otherwise $G_1$ would be a uniquely hamiltonian graph which is impossible since $G_1$ is cubic. Denote $f_{k-1} = e = uv$ such that $f_k$ is incident to $u$. Set $e_1 = f_k$, and let $e_2$ be the third edge incident to $u$. Now, there must be a hamiltonian cycle $C_1$ other than $C_0$ not containing $e_1$ since $e_1=f_k$ does not belong to all hamiltonian cycles of $G_1$. Thus $C_1$ traverses $e$ and $e_2$. This finishes the proof of Proposition~\ref{Pr:G1}.
 }\end{proof}

\section{NP-complete problems}\label{Sec:Complexity}

We now establish several NP-completeness results.\\
\noindent
In the proof of Theorem~\ref{CR:A-Tr.NP-C}, one may assume without loss of generality that the outer face and all quadrilaterals have color $2$.   Then by Observation~\ref{obs:A-trail}, we have the following corollary.
\begin{corollary}
Let $G$ be a Barnette graph with  a $3-$face-coloring   with color set $\{1,2,
3\}$. Assume the outer face
of $G$ is colored $2$  and $H$ is the reduced graph obtained by contracting the $1-$faces
and equipped with a $2-$face-coloring. Suppose that  $H$ has only triangles and quadrilaterals as face boundaries, and for which all quadrilaterals  have  color $2$ in the $2-$face-coloring. Then the decision problem of whether $H$ has a  quasi spanning tree defined by the set of all $($triangular$)$ $3-$faces
is NP-complete.
\end{corollary}
In Theorem~\ref{TH:4}, we give a similar result concerning the reduced graph $H$ containing only digons and quadrilaterals. For such a graph, we are {trying} to find a spanning tree of quadrilaterals.

The decision problem of whether a $3-$connected planar cubic graph $G_0$ has a hamiltonian
cycle is NP-complete, as shown by Garey et al.~\cite{Garey}.
Let $e =uv\in E(G_0)$. Then the decision problem of whether $G_0$ has a hamiltonian
cycle traversing this specified edge $e$, is also
NP-complete. Let $G_0^{'}=G_0\setminus \{e\}$. Thus, the decision problem
of whether $G_0^{'}$ has a hamiltonian path from $u$ to $v$ is also NP-complete.

\begin{theorem}\label{TH:4}
Let $G$ be a Barnette graph.
Let $c_f$ be a $3-$face-coloring  of $G$ with color set $\{1,2,
3\}$, and let $H$ be the reduced graph obtained by contracting the $1-$faces
and equipped with a $2-$face-coloring induced by $c_f$. Suppose that the $2-$faces
 in $H$ are quadrilaterals
and the $3-$faces in $H$
are digons. Then the decision problem of whether $H$ has a spanning tree of $2-$faces
is NP-complete.
\end{theorem}

\begin{proof}{ We want to construct $G$ and $H$ as stated  in the theorem. To this end,
consider $G_0$ and $G_0^{'}$ as described in the paragraph preceding the statement of Theorem~\ref{TH:4}
 and
assume $G_0^{'}$ is the plane graph resulting by edge deletion from a fixed imbedding of $G_0$.
Let $H$ be the plane graph resulting by replacing every  edge of the radial graph $\mathcal{R}(G_0^{'})$ with a digon; $H$ is eulerian and hence $2-$face-colorable.
First color
the digons corresponding to edges in $\mathcal{R}(G_0^{'})$ with color $3$. The remaining faces of $H$  are quadrilaterals
$Q=xfx^{'}f^{'}x$  corresponding to  $xx^{'}\in E(G_0^{'})\cap  bd(F) \cap bd(F^{'})$ with $F$ and $F^{'}$ in $G_0^{'}$ corresponding to $f,f^{'}\in V(\mathcal{R}(G_0^{'}))$. Color these quadrilaterals with color $2$.
Let $G$ be the plane cubic graph obtained from $H$ by replacing each $w\in V(H)$ with a cycle $C_w= w_1\ldots w_{\deg_H(w)}w_1$
and replacing $e_i= u_iw\in E(H)$ with $e_i^{'}=u_iw_i,$ for $i = 1,\ldots,deg_H(w)$.
We have $\kappa(H) \geq 2$, but $\kappa^{'}(H)>2$. Therefore, $G$ is $3-$connected and thus a Barnette graph whose $3-$face-coloring has color set $\{1,2,3\}$; the $1-$faces of $G$ correspond to the vertices of $H$.

\begin{claim}\label{clm:3}
A set $L$ of edges in $G_0^{'}$ forms a hamiltonian path from $u$ to $v$ in $G_0^{'}$ if and
only if $H_{\mathcal{T}}$  is a spanning tree of $2-$faces in $H$ where $\mathcal{T}$ is the set  of $2-$faces $($which are quadrilaterals$)$ in $H$ corresponding to the edges in $E(G_0^{'})\setminus L$.
\end{claim}

Suppose $L$ is a hamiltonian path from $u$ to $v$ in $G_0^{'}$. Let $L^{'}= E(G_0^{'})\setminus L$ (which is a perfect matching in both $G_0$ and $G_0^{'}$), and let $\mathcal{T}$ be the set of all  quadrilaterals in $H$ corresponding to $L^{'}$. Note that
since $L$  is a hamiltonian path, for any two edges $g,h\in L^{'}$, there is a sequence of edges $g = \ell_1,\ell_2, \ldots,\ell_k = h$ in $L^{'}$ such that each pair of
edges $\ell_i, \ell_{i+1}$ belongs to a face boundary in $G_0^{'}$, $1\le i\le k-1$. Therefore the $2-$faces in $\mathcal{T}$ induce a connected subgraph of $H$.
Notice also that every vertex in $H$ belongs to some face in $\mathcal{T}$ since every vertex $x\in V(G_0^{'})$ is
incident to an edge in $L^{'}$, and every face $F$ in $G_0^{'}$ contains at least one edge of $L^{'}$ in its boundary.

  Finally, there is no cycle of $2-$faces in $H_{\mathcal{T}}$.  Suppose to the contrary,  we had a cycle $Q_1Q_2\ldots Q_rQ_1$ of
$2-$faces in $H_{\mathcal{T}}$.  Since the number of $2-$faces in $H_{\mathcal{T}}$ containing $x$ is equal to $\deg_{G_0^{'}}(x)-\deg_L(x)=1$, for every vertex $x\in V(G_0^{'})$, so $Q_i$ and $Q_{i+1}$ share a vertex $f\in V(H)$  corresponding to  a face $F\in \mathcal{F}(G_0^{'})$. Thus
$\{q_1, q_2, \ldots, q_r\}\subset L^{'}$, with $q_i$ corresponding to the face $Q_i$ in the cycle of $2-$faces in $H_{\mathcal{T}}$, separates
the graph $G_0^{'}$ into at least two components; so the hamiltonian path $L$ would have to contain at least
one of these edges $q_i\in L^{'}$, a contradiction. Therefore, $H_{\mathcal{T}}$ is a spanning tree of $2-$faces for $H$.

Conversely, suppose $H_{\mathcal{T}}$ is a spanning tree of  $2-$faces for $H$. Let $L^{'}$ be the corresponding edges in $G_0^{'}$
(which appear as chords of the elements of $\mathcal{T}$ if one draws $G_0^{'}$ and $H$ in the plane as described before);
and let $L=E(G_0^{'})\setminus L^{'}$.
Each vertex $x\in V(G_0^{'})$  belongs to exactly one $2-$face $Q=xfx^{'}f^{'}x$ in $\mathcal{T}$, since every other
$2-$face in $\mathcal{T}$ containing $x$ also contains either $f$ or $f^{'}$, and therefore these two $2-$faces
share two vertices joined by parallel edges
 and thus cannot both be in the spanning tree of faces $H_{\mathcal{T}}$. Therefore every vertex in $G_0^{'}$ is incident to exactly one edge in $L^{'}$, and so the two vertices $u$ and $v$ of degree $2$ in $G_0^{'}$
are incident to exactly one edge in $L$, while the remaining vertices  of degree $3$ in $G_0^{'}$ are
incident to exactly two edges in $L$. That is, $L$ induces a path joining $u$ and $v$ in $G_0^{'}$ plus a possibly empty set
 of cycles in $G_0^{'}$, such that the path and the cycles are disjoint and cover all of $V(G_0^{'})$. We show that $L$ cannot contain a cycle in $G_0^{'}$, and so $L$ is just a hamiltonian path joining $u$
to~$v$.

Suppose
$L$ contains a cycle $C=h_1h_2\ldots h_kh_1$ in $G_0^{'}$. Let $F$ and $F^{'}$ be faces of $G_0^{'}$ inside and outside the
cycle $C$, respectively, and let $f$ and $f^{'}$ be the vertices in $H$ corresponding to  $F$ and $F^{'}$, respectively.  Since $f$ and $f^{'}$ are vertices in the spanning tree $H_{\mathcal{T}}$ of $2-$facces, there
is a unique sequence of $2-$faces $Q^*_1, Q^*_2,\ldots, Q^*_l$ in $\mathcal{T}$ such that $Q^*_1$ contains $f$, $Q^*_l$ contains $f^{'}$ and each
pair $Q^*_{i-1}$, $Q^*_i$ share a vertex $f_i$ corresponding to
a face in $G_0^{'}$, for $2\le i<l$. In particular, if we denote $f_1 = f$ and $f_l = f^{'}$, then for
some pair $f_i$, $f_{i+1}$,  for the
corresponding  face $F_i$ in $G_0^{'}$ we must have $F_i\subseteq G_0^{'} \cap int(C)$ and for the corresponding  face $F_{i+1}$ in $G_0^{'}$ we must have $F_{i+1}\subseteq G_0^{'}\cap ext(C)$. This implies that the  $2-$face $Q^*_i$ in $H_{\mathcal{T}}$ corresponds to one of the edges in $L$ and not in $L^{'}$, a contradiction. This completes the proof of Claim~\ref{clm:3}.

 Therefore by Claim~\ref{clm:3}, $H$ has a
spanning tree of $2-$faces if and only if $G_0^{'}$ has a hamiltonian path from $u$ to $v$, and so the
decision problem of whether $H$ has a spanning tree of $2-$faces is NP-complete.
}\end{proof}

We obtain two Corollaries from this result.

\begin{corollary}
Let $G$ be a Barnette graph
with a $3-$face-coloring with color set $\{1,2,
3\}$, and let $H^{'}$ be the reduced graph obtained by contracting the $1-$faces. Suppose all vertices of $H^{'}$ have degree $8$. Then the decision problem
of whether $H^{'}$ has an $A-$trail is NP-complete.
\end{corollary}

\begin{proof}{ Consider the reduced graph $H$
in the statement of Theorem~\ref{TH:4} where  all $2-$faces in $H$ are quadrilaterals,
corresponding to a facial $2-$factor of octagons in $G$. If we contract in $G$ these octagonal $2-$faces, we obtain an $8-$regular reduced graph
$H^{'}$. By Theorem~\ref{PR:3}, $H^{'}$ has an $A-$trail  if and only
if $H$ has a spanning tree of $2-$faces, and this problem is NP-complete by Theorem~\ref{TH:4}.
}\end{proof}

\begin{corollary}
Let $G$ be a Barnette graph
with a $3-$face-coloring with color set $\{1,2,
3\}$, and let $H_0$ be the reduced graph obtained by contracting the $1-$faces. Suppose that the $2-$faces in $H_0$ are octagons and digons and
the $3-$faces in $H_0$ are triangles. Then the decision problem of whether $H_0$ has a spanning tree of faces is NP-complete.
\end{corollary}

\begin{proof}{
Start with $G_0$ and $G_0^{'}$ as at the beginning of the proof of Theorem~\ref{TH:4}, and construct the reduced graph $H$ as in the proof
   of Theorem~\ref{TH:4}, with $2-$colored quadrilaterals and $3-$colored
digons. If $e$ and $f$ are the two parallel edges of a $3-$colored digon,
subdivide $e$ with  vertex $w$  and
subdivide $f$ with vertex $x$. Join  $w$ and $x$  by two parallel edges. The $3-$colored
digon splits thus into two $3-$colored triangles and a $2-$colored digon, while the $2-$colored quadrilaterals become
$2-$colored octagons, in the new reduced graph $H_0$.

Suppose $H$ has a spanning tree of $2-$colored quadrilaterals $H_{\mathcal{T}}$. Select the corresponding $2-$colored
octagons in $H_0$. For a $3-$colored digon consisting of two edges $e$ and $f$ in $H$, if one of the two $2-$colored
quadrilaterals containing $e$ or $f$ is in $\mathcal{T}$, then select the
$2-$colored digon joining the subdivision vertices
$w$ and $x$; if neither of the two $2-$colored quadrilaterals containing $e$ or $f$ is in $\mathcal{T}$, then select one
of the two $3-$colored triangles containing $w$ and $x$. The $2-$colored and
$3-$colored faces in $H_0$ thus selected,
involving $2-$colored octagons, $2-$colored digons, and $3-$colored triangles, form a spanning tree of faces in $H_0$.

Conversely, suppose $H_0$ has a spanning tree of faces $H_{0\mathcal{T}_0}$. Let $\mathcal{T}$ be the set of $2-$colored
quadrilaterals in $H$ such that the corresponding $2-$colored octagons are in $\mathcal{T}_0$. Note that for each
digon in $H$, at most one of the corresponding two $3-$colored triangles and $2-$colored digon in $H_0$ can be in $\mathcal{T}_0$. Thus $H_{\mathcal{T}}$ is a spanning tree of $2-$colored (quadrilateral) faces.

Note that the reduction process between these two decision problems  can be done in polynomial time,  since every
simple plane graph has $O(n)$ faces, where $n$ is the order of graph.

Thus $H_0$ has a spanning tree of arbitrary faces if and only if $H$ has a spanning tree of
$2-$colored faces, and NP-completeness follows from Theorem~\ref{TH:4}.
}\end{proof}

Finally we show that if Barnette's Conjecture is false, then it would be NP-complete to decide whether a Barnette graph is hamiltonian.

\begin{theorem}
Assume that  Barnette's Conjecture is false. Then the decision problem of whether a Barnette graph  has
a hamiltonian cycle, is NP-complete.
\end{theorem}

\begin{proof}{ Takanori et al.~\cite{Takanori} showed that the decision problem of whether a $2-$connected cubic
planar bipartite graph $R$ has a hamiltonian cycle is NP-complete.

 If such an $R$ has a $2-$edge-cut
$\{e_1,e_2\}$ that separates $R$ into two components $R^{'}$ and $R^{''}$, then their endpoints in either side
are at odd distance (see the above argument),
 so we may instead join the two endpoints of $e_1$ and $e_2$ in $R^{'}$ and $R^{''}$, separately, and
ask whether $R^{'}$ and $R^{''}$ both contain a hamiltonian cycle containing the added edge joining
the endpoints of $e_1$ and $e_2$.

Repeating this decomposition process, we eventually reduce the
decision problem of whether $R$ has a hamiltonian cycle to the decision problem of whether various $R_i$'s each
contain a hamiltonian cycle going through certain prespecified edges, with each $R_i$ being
$3-$connected or the cubic multigraph on $2$ vertices. Thus the decision problem of whether a Barnette graph $G^{'}$
has a hamiltonian cycle going through certain prespecified edges is NP-complete.

 Let a Barnette graph $G^{'}$
 with certain prespecified edges $e^{'}_1,\ldots,e^{'}_k$ that a hamiltonian cycle must traverse, be given.
Denote  $e_i^{'}=x_iy_{i,1}$ and $N_{G^{'}}(x_i)=\{y_{i,1},y_{i,2},y_{i,3}\}$, for $i=1,\ldots,k$.

Suppose that  Barnette's Conjecture is false. Then by
Proposition~\ref{Pr:G1}, there exists a hamiltonian Barnette graph $G_{i}$ with a vertex $u_i\in V(G_{i})$ and $N_{G_{i}}(u_i)=\{v_{i,1},v_{i,2},v_{i,3}\}$
 such that  every hamiltonian cycle
in $G_{i}$ traverses $e_i= u_iv_{i,1}\in E(G_i)$, $i=1\ldots,k$. Furthermore,   $G_i$ has a hamiltonian cycle traversing $e_i$ and $u_iv_{i,j}$, for $i=1,\ldots,k$ and $j=2,3$.

\medskip
Construct a new Barnette graph
$$G=\bigg(G^{'}\setminus \{x_1,\ldots,x_k\}\bigg)\cup
\bigg(\bigcup_{i=1}^k(G_i\setminus \{u_i\})\bigg) \cup \bigg(\bigcup_{i=1}^k\{v_{i,1}y_{i,1},v_{i,2}y_{i,2},v_{i,3}y_{i,3}\}\bigg).$$

Since every hamiltonian cycle in  $G_i$ traverses the edge $e_i$ and   $G_i$ has also a hamiltonian cycle traversing $e_i$ and $u_iv_{i,j}$, for $i=1,\ldots,k$ and $j=2,3$, the resulting graph $G$ has a hamiltonian cycle if and only if $G^{'}$
has a hamiltonian cycle traversing the edges $e_1^{'},\ldots,e_k^{'}$.
Moreover, $G$ can be constructed from $G^{'}$ in polynomial time and its vertex set is also polynomially enlarged from $G^{'}$.
 Therefore, the decision problem whether the resulting Barnette graph $G$ has a hamiltonian cycle, is NP-complete.
}\end{proof}



\end{document}